\documentclass[dvipdfmx,autodetect-engine]{journal}
\usepackage{tikz}
\usepackage{color}
\usepackage[all]{xy}
\usepackage{amsmath, amssymb}
\usepackage{amsfonts}
\usepackage{mathrsfs}
\usepackage{amsthm}
\usepackage{bm}
\usepackage{amscd}
\usepackage{tikz-cd}

\makeatletter
 \def\th@plain{\upshape}
 \makeatother

\newtheorem{theorem}{Theorem}[section]
\newtheorem{proposition}[theorem]{Proposition}
\newtheorem{lemma}[theorem]{Lemma}

\newtheorem{Example}[theorem]{Example}
\newtheorem{definition}[theorem]{Definition}
\newtheorem{rem}[theorem]{Remark}
\newtheorem{question}[theorem]{Question}

\title{Classification of Finite Subquandles in Spherical Quandles}
\author{
Kentaro Yonemura
\thanks{Osaka Central Advanced Mathematical Institute, Osaka Metropolitan University, Sugimoto, Sumiyoshi-ku, Osaka City 558-8585, Japan, E-mail: \texttt{yonemura-kentaro@sei.co.jp}}
\and
Naoki Nakamura
\thanks{Department of Mathematics, Kyushu University, 744 Motooka, Nishi-ku, Fukuoka 819–0395, Japan, E-mail: \texttt{nakamura.naoki.331@s.kyushu-u.ac.jp}}
}
\date{}

\usepackage{lineno}
\begin{document}

\maketitle
\begin{abstract}
 Every finite subquandle of a spherical quandle is accounted for by dihedral quandles and by quandles defined on root systems.
\end{abstract}



\section{Introduction}
Quandles are algebraic structures introduced independently by Joyce \cite{joyce1982classifying} and Matveev \cite{matveev1982distributive}, and are known to provide powerful invariants in knot theory. Since their introduction, quandles have been studied extensively not only in knot theory but also in topology, algebra, and geometry. For general background on quandles and related topics, see, for example, \cite{kamada2017surface,nosaka2017books}.

Joyce \cite{joyce1982classifying} proposed the viewpoint that quandles may be regarded as algebraic analogues of symmetric spaces. Indeed, every Riemannian symmetric space carries a natural quandle structure defined by its point symmetries.  From this perspective, there have also been attempts to study quandles as generalizations or discretizations of symmetric spaces.

A notable example arising from this perspective is the study of spherical quandles. Azcan--Fenn \cite{AzcanFenn1994} considered the quandle structure on the \(n\)-dimensional unit sphere \(S^n\) induced by its structure as a Riemannian symmetric space. Moreover, they studied this quandle together with its natural extensions, which they called spherical quandles. Subsequently, \"{O}zdemir--Azcan \cite{Azcan2010finite} investigated finite subquandles of spherical quandles and provided a complete classification for \(S^1\) and \(S^2\). In particular, their classification for \(S^2\) essentially relies on the classification of finite subgroups of the real orthogonal group \(O(3)\). Although this approach is effective in low dimensions, its complexity increases rapidly when one attempts to extend it directly to higher dimensions. Consequently, the classification of finite subquandles of \(S^n\) for general \(n\) has remained open.

The aim of this paper is to resolve this problem in all dimensions. Our main result is the following theorem.

\begin{theorem}
\label{main_theorem}
Let \(X\) be a finite subquandle of the spherical quandle \(S^{n}_{\mathbb{R}}\). If \(\#X\) is even, then \(X\) is a root system in \(\mathbb{R}^{n+1}\) in the sense of Definition~\ref{definition_root_system}; if \(\#X\) is odd, then \(X\) is isomorphic to the dihedral quandle \(R_{\#X}\).
\end{theorem}

This theorem completely classifies the finite subquandles of spherical quandles: those of even cardinality are precisely finite root systems consisting of unit vectors, whereas those of odd cardinality are precisely dihedral quandles. In particular, finite subquandles of spherical quandles exhibit a striking rigidity: their structure is completely determined by the parity of their cardinality. Note that every finite root system consisting of unit vectors gives rise to a finite subquandle of a spherical quandle. Therefore, every isomorphism class of finite root systems arises in this way. To the best of our knowledge, Theorem~\ref{main_theorem} gives the first concise quandle-theoretic characterization of finite root systems. Our proof does not rely on the classification of finite subgroups of \(O(n+1)\), and thus provides a uniform treatment in all dimensions. The proof is based on the observation that the presence of one antipodal pair forces antipodal closure, whereas the absence of antipodal pairs forces the subquandle to lie in a \(2\)-dimensional subspace, reducing the problem to the classification of finite subquandles of the spherical quandle on \(S^{1}\).

The problem studied in \cite{Azcan2010finite} and in this paper is a special case of the following, more general question.

\begin{question}
\label{question_finite_subquandle_Riemann_symmetric_spaces}
Let \(X\) be a compact Riemannian symmetric space equipped with the natural quandle structure defined by its point symmetries. Classify the finite subquandles of \(X\).
\end{question}
Although Question~\ref{question_finite_subquandle_Riemann_symmetric_spaces} remains open in general, Theorem~\ref{main_theorem} provides the first answer in the case of spheres. In the case of spherical quandles, the behavior of antipodal points plays an essential role. This suggests that the notion of antipodal sets, introduced by Chen--Nagano~\cite{ChenNagano1988riemannian}, may be useful in studying Question~\ref{question_finite_subquandle_Riemann_symmetric_spaces}. More generally, one can formulate an analogous question for homogeneous quandles. For background on homogeneous quandles, see \cite{IshiharaTamaru2016flat}.

We conclude this introduction by describing the organization of the paper. In Section~\ref{section_preliminaries}, we recall the basic notions of quandle theory needed in this paper. In Section~\ref{section_One-dimensional_spherical_quandle}, we review the properties of the spherical quandle on \(S^{1}\) relevant to this paper and present \"{O}zdemir--Azcan's classification of its finite subquandles. In Section~\ref{section_Root_systems_as_quandles}, we introduce and study a quandle structure on root systems. Finally, in Section~\ref{section_Classification_of_finite_subquandles_of_spherical_quandles}, we prove Theorem~\ref{main_theorem}.

\section{Preliminaries}
\label{section_preliminaries}
In this section, we review the terminology on quandles used in this paper. See \cite{kamada2017surface,nosaka2017books} for more details.

A \emph{quandle} is a pair \((X,\triangleright)\) consisting of a nonempty set \(X\) and a binary operation \(\triangleright : X \times X \to X\) satisfying the following three conditions:
\begin{enumerate}
  \item[Q1] For every \(x \in X\), one has \(x \triangleright x = x\).
  \item[Q2] For each \(y \in X\), the map \(s_y : X \to X\) defined by \(x \mapsto x \triangleright y\) is a bijection.
  \item[Q3] For all \(x,y,z \in X\), one has \((x \triangleright y)\triangleright z = (x \triangleright z)\triangleright (y \triangleright z)\).
\end{enumerate}
The conditions Q1, Q2, and Q3 are consistent with Reidemeister moves, operations for knot diagrams in knot theory, I, II and III respectively. See Figure \ref{pic_def_quandle_operation}.

\begin{figure}[hbtp]
\centering
\includegraphics[width=12cm]{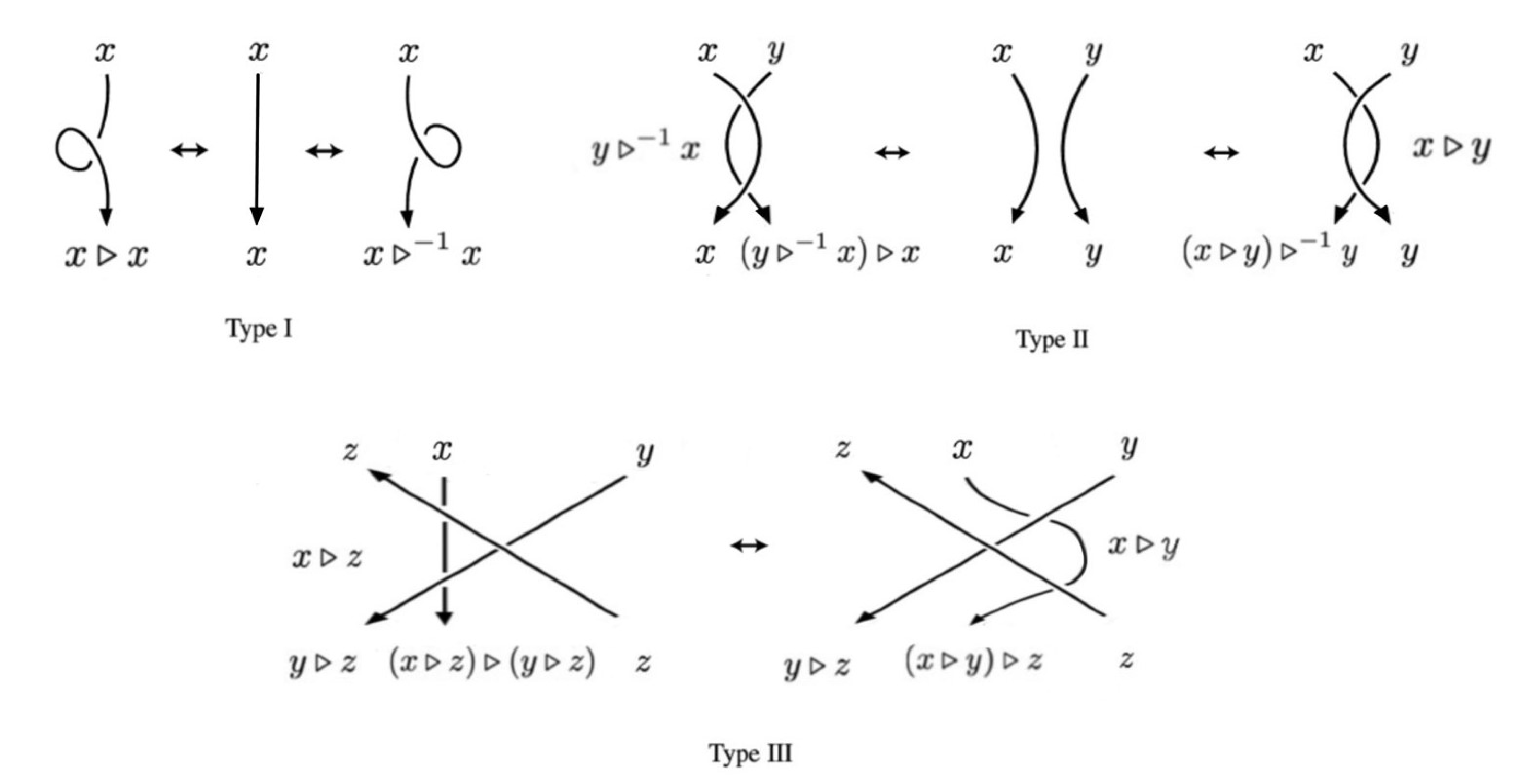}
\caption{Geometric interpretation of Q1, Q2 and Q3.}
\label{pic_def_quandle_operation}

\end{figure}

We now introduce the quandles considered in this paper.

\begin{Example}[Trivial quandle]
For any nonempty set \(X\), define a binary operation \(\triangleright : X \times X \to X\) by
\[
x \triangleright y := x
\qquad
(x,y \in X).
\]
Then \((X,\triangleright)\) is a quandle. This is called the \emph{trivial quandle}.
\end{Example}

\begin{Example}[Dihedral quandle~\cite{takasaki1943abstraction}]
Let \(M=\mathbb{Z}/n\mathbb{Z}\) be the cyclic group of order \(n\). Define a binary operation \(\triangleright : M \times M \to M\) by
\[
x \triangleright y := 2y-x
\qquad
(x,y \in M).
\]
Then \((M,\triangleright)\) is a quandle. This is called the \emph{dihedral quandle} and is denoted by \(R_n\).
\end{Example}

\begin{Example}[Spherical quandle~\cite{AzcanFenn1994}]
Consider the unit sphere
\[
S^n
=
\left\{
(x_1,\dots,x_{n+1}) \in \mathbb{R}^{n+1}
\;\middle|\;
x_1^2+\cdots+x_{n+1}^2=1
\right\}.
\]
Let \(\langle -,-\rangle\) denote the Euclidean inner product on \(\mathbb{R}^{n+1}\). Define a binary operation \(\triangleright : S^n \times S^n \to S^n\) by
\[
x \triangleright y := 2\langle x,y\rangle y-x
\qquad
(x,y \in S^n).
\]
Then \((S^n,\triangleright)\) is a quandle. This is called the \emph{spherical quandle} and is denoted by \(S^n_{\mathbb{R}}\).
\end{Example}

The spherical quandle is an example of the quandle structure on a Riemannian symmetric space defined by point symmetries, as observed by Joyce. In this paper, we shall occasionally use without further comment both the fact that \(S^n\) is a Riemannian symmetric space and the fact that \(S^n\) is naturally embedded in the Euclidean space \(\mathbb{R}^{n+1}\).

Let \(X\) and \(Y\) be quandles. A map \(f:X\to Y\) is called a \emph{quandle homomorphism} if \(
f(x\triangleright y)=f(x)\triangleright f(y)
\) for \(x,y\in X\). A bijective quandle homomorphism is called a \emph{quandle isomorphism}. By definition, every quandle gives rise to distinguished quandle automorphisms: for each \(y\in X\), one has a map\(s_y:X\to X\). Condition (Q3) implies that \(s_y\) is a quandle homomorphism, and by (Q2), it is bijective. Hence \(s_y\) is a quandle automorphism. This automorphism \(s_y\) is called the \emph{inner automorphism} associated with \(y\). We write \(\operatorname{Inn}(X)\) for the subgroup of the automorphism group of \(X\) generated by these inner automorphisms, and call it the \emph{inner automorphism group} of \(X\). A \emph{subquandle} of a quandle \((X,\triangleright)\) is a subset of \(X\) on which the operation \(\triangleright\) induces a quandle structure.

In this paper, we write \(\#S\) for the cardinality of a set \(S\). 

\section{One-dimensional spherical quandle}
\label{section_One-dimensional_spherical_quandle}
In this section, we review the one-dimensional spherical quandle \(S^1_{\mathbb{R}}\) and present Özdemir and Azcan's results on the finite subquandles of \(S^1_{\mathbb{R}}\).

The circle \(S^{1}\) admits the following parametrization:
\[
S^{1}=\{(\cos\theta,\sin\theta)\in\mathbb{R}^{2}\;:\;\theta\in\mathbb{R}\}.
\]
When \(S^{1}\) is endowed with the structure of a spherical quandle, its quandle operation is given as follows, and the corresponding inner automorphism is realized by reflection.
\begin{equation}
\label{prop_operator_S^1spherical_quandle}
(\cos{\alpha},\sin{\alpha})
\triangleright
(\cos{\beta},\sin{\beta})
=
(\cos{(2\beta-\alpha)},\sin{(2\beta-\alpha)})
=
(\cos{\alpha},\sin{\alpha})
\begin{pmatrix}
\cos{2\beta} & \sin{2\beta}\\
\sin{2\beta} & -\cos{2\beta}
\end{pmatrix}
.
\end{equation}
It follows that the inner automorphism group \(\operatorname{Inn}S^{1}_{\mathbb{R}}\) of the spherical quandle \(S^{1}_{\mathbb{R}}\) is the orthogonal group \(O(2)\).

\begin{proposition}
\label{prop_lotation_auto_S^1_R}
For each \(\theta \in \mathbb{R}\), the isometry $C_{\theta}: S^{1} \to S^{1}$ defined by
\[
\qquad
C_{\theta}(\bm{x}):= 
\bm{x}
\begin{pmatrix}
\cos\theta & \sin\theta\\
-\sin\theta & \cos\theta
\end{pmatrix}
\quad
(\bm{x}\in S^{1})
\]
is a quandle automorphism of \(S^{1}_{\mathbb{R}}\).
\end{proposition}
\begin{proof}
The composition of two reflections is a rotation. Indeed, for any \(\alpha,\beta\in\mathbb{R}\), we have
\[
\begin{pmatrix}
\cos\alpha & \sin\alpha\\
\sin\alpha & -\cos\alpha
\end{pmatrix}
\begin{pmatrix}
\cos\beta & \sin\beta\\
\sin\beta & -\cos\beta
\end{pmatrix}
=
\begin{pmatrix}
\cos(\beta-\alpha) & \sin(\beta-\alpha)\\
-\sin(\beta-\alpha) & \cos(\beta-\alpha)
\end{pmatrix}.
\]
Therefore, by \eqref{prop_operator_S^1spherical_quandle}, if we choose \(\alpha,\beta\) so that \(\theta=2(\beta-\alpha)\), then \(C_{\theta}=s_{\beta}\circ s_{\alpha}\).
\end{proof}
We shall show in Proposition \ref{prop_Ozdemir-Azcan} that the spherical quandle \(S^{1}_{\mathbb{R}}\) has only very restricted subquandles. Although this is essentially contained in the discussion at the beginning of Section 3 of \cite{Azcan2010finite}, we include a proof here for the reader's convenience and for the sake of completeness, since the result plays an important role in what follows. We begin with some preparation. Consider the subset
\begin{equation}
\label{prop_R_ncong_mathcalR_n}
\mathcal{R}_{n}
=
\left\{
\left(\cos\frac{2\pi k}{n},\sin\frac{2\pi k}{n}\right)
\in S^{1}
\;:\;
k=0,1,2,\ldots,n-1
\right\}
\end{equation}
of \(S^{1}\). It follows from \eqref{prop_operator_S^1spherical_quandle} that \(\mathcal{R}_{n}\) is a subquandle of \(S^{1}_{\mathbb{R}}\). Define a map \(f \colon R_{n} \to \mathcal{R}_{n}\) by
\[
f(k+n\mathbb{Z})
:=
\left(\cos\frac{2\pi k}{n},\sin\frac{2\pi k}{n}\right)
\qquad
(k\in\mathbb{Z}).
\]
Then, by the periodicity of the trigonometric functions, \(f\) is a well-defined bijection. Moreover, it follows from \eqref{prop_operator_S^1spherical_quandle} that \(f\) is a quandle isomorphism.

For \(\alpha,\beta \in \mathbb{R}\), define
\[
M(\alpha,\beta)
:=
\left\{
\bigl(\cos(k(\beta-\alpha)+\alpha),\ \sin(k(\beta-\alpha)+\alpha)\bigr)\in S^1
\;\middle|\;
k\in\mathbb{Z}
\right\}.
\]
\begin{lemma}
\label{lemma_M(a,b)_subquandle}
Let \(X\) be a subquandle of \(S^{1}_{\mathbb{R}}\), and let \(\alpha,\beta \in \mathbb{R}\). If \((\cos\alpha,\sin\alpha)\) and \((\cos\beta,\sin\beta)\) belong to \(X\), then \(M(\alpha,\beta)\) is a subquandle of \(X\).
\end{lemma}
\begin{proof}
Define a sequence \(\{a_k\}_{k\in\mathbb{Z}}\) recursively by
\[
a_k=
\begin{cases}
(\cos\alpha,\sin\alpha) & (k=0),\\
(\cos\beta,\sin\beta) & (k=1),\\
a_{k-2}\triangleright a_{k-1} & (k\ge 2),\\
a_{k+2}\triangleright a_{k+1} & (k\le -1).
\end{cases}
\]
A straightforward induction shows that, for every \(k\in\mathbb{Z}\),
\[
a_k=
\bigl(\cos(k(\beta-\alpha)+\alpha),\sin(k(\beta-\alpha)+\alpha)\bigr).
\]
Hence \(a_k\in X\) for all \(k\in\mathbb{Z}\). Since 
\(
M(\alpha,\beta)=\{a_k : k\in\mathbb{Z}\},
\) 
it follows from \eqref{prop_operator_S^1spherical_quandle} that \(M(\alpha,\beta)\) is a subquandle of \(X\).
\end{proof}

\begin{lemma}
\label{lemma_Nakamura_Naoki}
Let $\alpha,\beta\in\mathbb{R}$, and let $n$ be a positive integer. If $M(\alpha,\beta)$ is finite, then $\beta-\alpha\in 2\pi\mathbb{Q}$. In particular, if $\#M(\alpha,\beta)=n$, then
\[
M(\alpha,\beta)
=
\left\{
\left(\cos(\alpha+\tfrac{2\pi k}{n}),\sin(\alpha+\tfrac{2\pi k}{n})\right)
\ :\ 
k=0,1,2,\dots,n-1
\right\}
= C_{\alpha}(\mathcal{R}_{n}).
\]
\end{lemma}
\begin{proof}
Since $M(\alpha,\beta)$ is finite, there exist distinct integers $l$ and $m$ such that
\[
(\cos(l(\beta-\alpha)+\alpha),\sin(l(\beta-\alpha)+\alpha))
=
(\cos(m(\beta-\alpha)+\alpha),\sin(m(\beta-\alpha)+\alpha)).
\]
Hence,
\[
(l(\beta-\alpha)+\alpha)-(m(\beta-\alpha)+\alpha)
=
(l-m)(\beta-\alpha)
\in 2\pi\mathbb{Z}.
\]
Therefore, $\beta-\alpha\in 2\pi\mathbb{Q}$.

Write 
\(
\beta-\alpha = \frac{2\pi p}{q},
\) 
where $p$ and $q$ are relatively prime integers. Then there exist integers $a$ and $b$ such that \(ap+bq=1.\) Therefore, for any integer $k$,
\begin{equation*}
\frac{2\pi k}{q}
=
2\pi ka\frac{p}{q}+2\pi bk
=
ka(\beta-\alpha)+2\pi bk.
\end{equation*}
It follows that
\[
M(\alpha,\beta)
=
\left\{
\left(
\cos\left(\alpha+\tfrac{2\pi k}{q}\right),\sin\left(\alpha+\tfrac{2\pi k}{q}\right)
\right)
\ :\ 
k=0,1,2,\dots,q-1
\right\}.
\]
Thus, $M(\alpha,\beta)=C_{\alpha}(\mathcal{R}_{q})$ and $\#M(\alpha,\beta)=q$, and the assertion follows.
\end{proof}

\begin{proposition}[\cite{Azcan2010finite}]
\label{prop_Ozdemir-Azcan}
Let $X$ be a finite subquandle of the spherical quandle $S^{1}_{\mathbb{R}}$, and set $n=\#X$. Then there exists $\theta\in\mathbb{R}$ such that $X=C_{\theta}(\mathcal{R}_{n}).$ In particular, every finite subquandle of the spherical quandle $S^{1}_{\mathbb{R}}$ is isomorphic to a dihedral quandle.
\end{proposition}
\begin{proof}
Each element of $X$ can be written uniquely in the form $(\cos \alpha,\sin \alpha)$ for some $\alpha\in[0,2\pi)$. If $\#X=1$, writing \(X=\{(\cos\alpha,\sin\alpha)\}\), we have $X=C_{\alpha}(\mathcal{R}_1)$. Hence it remains to consider the case $\#X\ge 2$.

Define $\theta\in[0,2\pi)$ by
\[
\theta
:=
\min\left\{
\alpha\in[0,2\pi)
\;\middle|\;
(\cos\alpha,\sin\alpha)\in X
\right\}.
\]
Since $\#X\ge 2$, we may further define $\phi\in(0,2\pi)$ by
\[
\phi
:=
\min\left(
\left\{
\alpha\in[0,2\pi)
\;\middle|\;
(\cos\alpha,\sin\alpha)\in X
\right\}\setminus\{\theta\}
\right).
\]
We shall prove by contradiction that $X=M(\theta,\phi)$. By Lemma~\ref{lemma_M(a,b)_subquandle}, it is enough to assume that \(X\setminus M(\theta,\phi)\neq\emptyset\) and derive a contradiction.

Thus we may define $\psi\in[0,2\pi)$ by
\[
\psi
:=
\min\left\{
\alpha\in[0,2\pi)
\;\middle|\;
(\cos\alpha,\sin\alpha)\in X\setminus M(\theta,\phi)
\right\}.
\]
Since $\psi>\phi>\theta$, there exist a positive integer $m\in\mathbb{Z}$ and a real number $r$ with $0<r<\phi-\theta$ such that\(\psi-\theta=m(\phi-\theta)+r\). Then
\begin{eqnarray*}
&&(\cos(2r+\theta),\sin(2r+\theta))\\
&=&
(\cos(2\psi-2m(\phi-\theta)-\theta),\sin(2\psi-2m(\phi-\theta)-\theta))\\
&=&
(\cos(2m(\phi-\theta)+\theta),\sin(2m(\phi-\theta)+\theta))
\triangleright
(\cos\psi,\sin\psi)
\in X.
\end{eqnarray*}
Note also that \(0<2r<2(\phi-\theta)\).

\noindent\textbf{Case 1: $0<2r<\phi-\theta$.}
In this case, \(\theta<2r+\theta<\phi\), and \((\cos(2r+\theta),\sin(2r+\theta))\in X\). This contradicts the definition of $\phi$.

\medskip
\noindent\textbf{Case 2: $\phi-\theta=2r$.}
Then $\psi-\phi+\theta\in(\theta,\psi)$. Moreover,
\[
\psi-\phi+\theta
=
\psi-2r
=
\psi-2(\psi-\theta-m\phi+m\theta)
=
2(m(\phi-\theta)+\theta)-\psi.
\]
Hence, by \eqref{prop_operator_S^1spherical_quandle},
\[
(\cos(\psi-\phi+\theta),\sin(\psi-\phi+\theta))
=
(\cos\psi,\sin\psi)\triangleright
(\cos(m(\phi-\theta)+\theta),\sin(m(\phi-\theta)+\theta))
\in X.
\]

We next claim that
\[
(\cos(\psi-\phi+\theta),\sin(\psi-\phi+\theta))\notin M(\theta,\phi).
\]
Indeed, suppose otherwise. Then there exists an integer $l$ such that
\[
(\cos(\psi-\phi+\theta),\sin(\psi-\phi+\theta))
=
(\cos(l\phi-(l-1)\theta),\sin(l\phi-(l-1)\theta)).
\]
Therefore, \((\psi-\phi+\theta)-(l\phi-(l-1)\theta)\in2\pi\mathbb{Z}\). Since
\[
(\psi-\phi+\theta)-(l\phi-(l-1)\theta)
=
\psi-((l+1)(\phi-\theta)+\theta),
\]
it follows that
\[
(\cos\psi,\sin\psi)
=
(\cos((l+1)(\phi-\theta)+\theta),\sin((l+1)(\phi-\theta)+\theta))
\in M(\theta,\phi),
\]
contradicting the definition of $\psi$.

Thus we have \(\theta<\psi-\phi+\theta<\psi\) and \((\cos(\psi-\phi+\theta),\sin(\psi-\phi+\theta))\in X\setminus M(\theta,\phi)\), which contradicts the definition of $\psi$.

\medskip
\noindent\textbf{Case 3: $\phi-\theta<2r<2(\phi-\theta)$.}
In this case, \(\theta<2\phi-2r-\theta<\phi\), and by \eqref{prop_operator_S^1spherical_quandle},
\[
(\cos(2\phi-2r-\theta),\sin(2\phi-2r-\theta))
=
(\cos(2r+\theta),\sin(2r+\theta))
\triangleright
(\cos\phi,\sin\phi)
\in X.
\]
This again contradicts the definition of $\phi$.

Therefore, in every case we obtain a contradiction. Hence $X=M(\theta,\phi)$. Since $X$ is finite, Lemma~\ref{lemma_Nakamura_Naoki} implies that \(X=C_{\theta}(\mathcal{R}_n).\)
\end{proof}

\begin{rem}
\label{rem_antipodaipoint_subquandle_S^1}
The quandle $\mathcal{R}_n$ has the following property:
\begin{itemize}
    \item if $n$ is even, then $\mathcal{R}_n$ is closed under taking antipodes;
    \item if $n$ is odd, then no antipode of a point of $\mathcal{R}_n$ belongs to $\mathcal{R}_n$.
\end{itemize}
By Proposition~\ref{prop_Ozdemir-Azcan}, every finite subquandle of $S^{1}_{\mathbb{R}}$ also satisfies the above property.
\end{rem}

\section{Root systems as quandles}
\label{section_Root_systems_as_quandles}
In this section, after recalling the definition of a root system, we show that any root system consisting of unit vectors forms a subquandle of a spherical quandle. For root systems, we refer the reader to \cite{humphreys1992reflection}.

Let \(n\) be a positive integer, and consider the Euclidean space \(\mathbb{R}^n\). Let \(\langle -,-\rangle : \mathbb{R}^n \times \mathbb{R}^n \to \mathbb{R}\) 
denote the Euclidean inner product. For \(\alpha \in \mathbb{R}^n \setminus \{0\}\), define the hyperplane orthogonal to \(\alpha\) by 
\(
H_\alpha
=
\left\{
v \in \mathbb{R}^n \;\middle|\;
\langle v,\alpha\rangle = 0
\right\}.
\) Then the reflection with reflecting hyperplane \(H_\alpha\) is defined to be the linear transformation \(\sigma_\alpha : \mathbb{R}^n \to \mathbb{R}^n\) given by
\[
\sigma_\alpha(v)
:=
v-\frac{2\langle v,\alpha\rangle}{\langle \alpha,\alpha\rangle}\alpha
\qquad
(v\in \mathbb{R}^n).
\]
This reflection is an involution and preserves the Euclidean inner product.
\begin{definition}
\label{definition_root_system}
A finite subset \(X \subset \mathbb{R}^n \setminus \{0\}\) is called a \emph{root system} if it satisfies the following conditions:
\begin{enumerate}
    \item For every \(\alpha \in X\), one has \(\mathbb{R}\alpha \cap X = \{\pm \alpha\}\).
    \item For every \(\alpha \in X\), one has \(\sigma_\alpha(X)=X\).
\end{enumerate}
\end{definition}

\begin{proposition}[Root systems as subquandles of the spherical quandle]
\label{prop_suqunadle_root_system}
Let \(X\) be a root system in the Euclidean space \(\mathbb{R}^{n+1}\), and assume that \(X\) is a subset of the unit sphere \(S^n\). Then \(X\) is a finite subquandle of the spherical quandle \(S^n_{\mathbb{R}}\).
\end{proposition}
\begin{proof}
For each \(\alpha \in X\), note that the inner automorphism \(s_\alpha\) of the spherical quandle can be written as \(-\sigma_\alpha\) in terms of the corresponding reflection. The assertion then follows immediately from the definition of a root system and the basic properties of reflections.
\end{proof}

\section{Classification of finite subquandles of spherical quandles}
\label{section_Classification_of_finite_subquandles_of_spherical_quandles}
In the preceding sections, we have considered the dihedral quandle \(\mathcal{R}_n\) and the quandles arising from root systems as examples of finite subquandles of spherical quandles. In this section, we show that, up to isomorphism, these exhaust all finite subquandles of spherical quandles. This proves Theorem~\ref{main_theorem}.

Let us first consider the cases in which the cardinality is very small.
\begin{proposition}[The cases \(\#X=1\) and \(\#X=2\)]
\label{prop_caseX=1or2}
Let \(X\) be a finite subquandle of the spherical quandle \(S^n_{\mathbb{R}}\). If \(\#X=1\), then \(X\) is isomorphic to the trivial quandle \(R_1\). Moreover, \(\#X=2\) if and only if there exists \(x\in S^n_{\mathbb R}\) such that \(X=\{\pm x\}\).
\end{proposition}
\begin{proof}
If \(\#X=1\), then the assertion is immediate, since there is only one quandle structure on a singleton.

Now suppose that \(\#X=2\). Write \(X=\{x,y\}\). In general, any quandle of cardinality \(2\) is necessarily trivial. Since \(x\triangleright y=x\), we obtain \(x=\langle x,y\rangle y.\) Similarly, since \(y\triangleright x=y\), we obtain \(y=\langle x,y\rangle x.\) As \(x\neq y\), it follows that necessarily \(y=-x\) and \(\langle x,y\rangle=-1\). Conversely, if \(X=\{\pm x\}\), then one checks directly that \(X\) is indeed a subquandle of \(S^n_{\mathbb{R}}\).
\end{proof}

By Proposition~\ref{prop_caseX=1or2}, it suffices to consider the case where the cardinality is at least \(3\).

We introduce some notation. Let \(n\) and \(k\) be positive integers with \(n \geq k\). For an \(n\)-dimensional real vector space \(V\), let \(\operatorname{Gr}_k(V)\) denote the set of all \(k\)-dimensional linear subspaces of \(V\).

\begin{proposition}
\label{prop_subspaceV_cap_subquandle}
Let \(X\) be a subquandle of the spherical quandle \(S^n_{\mathbb{R}}\). Let \(k\) be an integer with \(2 \leq k \leq n+1\), and let \(V \in \operatorname{Gr}_k(\mathbb{R}^{n+1})\). Then \(X \cap V\) is a subquandle of \(S^n_{\mathbb{R}}\). In particular, \(S^n \cap V\) is a subquandle of \(S^n_{\mathbb{R}}\) isomorphic to \(S^{k-1}_{\mathbb{R}}\), and \(X \cap V\) is a subquandle of \(S^n_{\mathbb{R}} \cap V\).
\end{proposition}
\begin{proof}
Let \(x,y \in X \cap V\), and let \(\triangleright\) denote the quandle operation on \(S^n_{\mathbb{R}}\). Since \(X\) is a subquandle of \(S^n_{\mathbb{R}}\), we have \(x \triangleright y \in X\). On the other hand, \(x \triangleright y = 2\langle x,y\rangle y - x\) is a linear combination of \(x\) and \(y\), and hence belongs to \(V\). Therefore \(x \triangleright y \in X \cap V\), which shows that \(X \cap V\) is a subquandle of \(S^n_{\mathbb{R}}\).

Now consider the case \(X = S^n_{\mathbb{R}}\). Then \(S^n_{\mathbb{R}} \cap V\) is also a subquandle of \(S^n_{\mathbb{R}}\). The Euclidean inner product on \(\mathbb{R}^{n+1}\) induces an inner product on \(V\). Choose an orthonormal basis \(v_1,\dots,v_k\) of \(V\) with respect to this inner product, and let \(e_1,\dots,e_k\) be the standard orthonormal basis of the Euclidean space \(\mathbb{R}^k\). Consider the linear isomorphism \(i_k^n : \mathbb{R}^k \to V\) induced by the correspondence \(e_i \mapsto v_i\) for each \(i \in \{1,\dots,k\}\). Then \(i_k^n\) preserves the inner product. It follows that \(i_k^n\) induces a quandle isomorphism \(S^{k-1}_{\mathbb{R}} \cong S^n_{\mathbb{R}} \cap V\).
\end{proof}
\begin{rem}
\label{rem_antipodalpoint_subquandle_S^n}
Consider the case \(k=2\) in Proposition~\ref{prop_subspaceV_cap_subquandle}. Then \(X \cap V\) can be identified with a subquandle of \(S^1_{\mathbb{R}}\). In this case, by the choice of \(i_k^n\) made in the proof of Proposition~\ref{prop_subspaceV_cap_subquandle}, the result of Proposition~\ref{prop_Ozdemir-Azcan} and the property of antipodal points described in Remark~\ref{rem_antipodaipoint_subquandle_S^1} carry over. More precisely, \(X \cap V\) is isomorphic to \(\mathcal{R}_{\#(X \cap V)}\), and the following hold:
\begin{itemize}
    \item \(\#(X \cap V)\) is even if and only if the antipode of each point of \(X \cap V\) also belongs to \(X \cap V\);
    \item \(\#(X \cap V)\) is odd if and only if the antipode of each point of \(X \cap V\) does not belong to \(X \cap V\).
\end{itemize}
\end{rem}

\begin{proposition}
\label{prop_sum_subquandle_elements=0}
Let \(X\) be a finite subquandle of the spherical quandle \(S^n_{\mathbb{R}}\). Define \(\Sigma_X:=\sum_{y\in X} y \in \mathbb{R}^{n+1}\). If \(\#X \geq 2\), then \(\Sigma_X=0\).
\end{proposition}
\begin{proof}
By Proposition~\ref{prop_caseX=1or2}, the assertion holds when \(\#X=2\). We therefore consider the case \(\#X \geq 3\). Choose \(x \in X\). Since \(s_x(X)=X\), we have \(s_x(\Sigma_X)=\Sigma_X\). On the other hand,
\[
s_x(\Sigma_X)
=\sum_{y\in X} s_x(y)
=\sum_{y\in X} \bigl(2\langle x,y\rangle x-y\bigr)
=2\sum_{y\in X} \langle x,y\rangle x-\Sigma_X.
\]
Hence \(\Sigma_X=\sum_{y\in X} \langle x,y\rangle x \in \mathbb{R}x\). Since we are assuming \(\#X \geq 3\), we can choose \(z \in X\setminus\{\pm x\}\). By the same argument, we obtain \(\Sigma_X=\sum_{y\in X} \langle z,y\rangle z \in \mathbb{R}z\). Because \(x\) and \(z\) are linearly independent in \(\mathbb{R}^{n+1}\), it follows that \(\Sigma_X=0\).
\end{proof}

Although the following proposition is proved in \cite[Proposition 8]{Azcan2010finite}, we include a proof for the convenience of the reader and for completeness, since this result plays an important role in what follows.
\begin{proposition}[\cite{Azcan2010finite}]
\label{prop_antipodalpoint_subquandle}
Let \(X\) be a finite subquandle of the spherical quandle \(S^n_{\mathbb{R}}\). If \(-x \in X\) for some \(x \in X\), then \(X\) contains the antipode \(-y\) of every element \(y \in X\).
\end{proposition}
\begin{proof}
By Proposition~\ref{prop_caseX=1or2}, the assertion holds when \(\#X=1\) or \(2\). We therefore consider the case \(\#X \geq 3\). Let \(x \in X\) be an element such that \(-x \in X\). If \(y \in X \setminus \{\pm x\}\), then \(x\) and \(y\) are linearly independent in \(\mathbb{R}^{n+1}\). Let \(\Pi\) be the \(2\)-dimensional linear subspace of \(\mathbb{R}^{n+1}\) spanned by \(x\) and \(y\). Note that \(\pm x, y \in \Pi\). By Remark~\ref{rem_antipodalpoint_subquandle_S^n}, the antipode of every element of \(X \cap \Pi\) belongs to \(X \cap \Pi\).
\end{proof}
\begin{proposition}
\label{prop_main1_root_system}
Let \(X\) be a finite subquandle of the spherical quandle \(S^n_{\mathbb{R}}\). If \(-x \in X\) for some \(x \in X\), then \(X\) is a root system in \(\mathbb{R}^{n+1}\).
\end{proposition}
\begin{proof}
By Proposition~\ref{prop_antipodalpoint_subquandle}, we have \(-x \in X\) for every \(x \in X\). Since \(X \cap \mathbb{R}x \subset S^n \cap \mathbb{R}x = \{\pm x\}\), it follows that \(X \cap \mathbb{R}x = \{\pm x\}\). Hence \(-X = X\). Since \(X\) is a subquandle of \(S^n_{\mathbb{R}}\), we have \(s_x(X)=X\). As \(s_x=-\sigma_x\), it follows that \(\sigma_x(X)=-s_x(X)=-X=X\). Therefore, \(X\) satisfies the conditions in Definition~\ref{definition_root_system}, and hence is a root system.
\end{proof}

\begin{proposition}
\label{prop_main2_dihedral_quandle}
Let \(X\) be a finite subquandle of the spherical quandle \(S^n_{\mathbb{R}}\). If \(-x \notin X\) for every \(x \in X\), then \(\#X\) is odd and \(X\) is isomorphic to the dihedral quandle \(R_{\#X}\).
\end{proposition}
\begin{proof}
By Proposition~\ref{prop_caseX=1or2}, the cases \(\#X=1\) and \(\#X=2\) have already been established.

We now consider the case \(\#X \geq 3\). Fix \(x \in X\), and define the set \(\mathcal{P}_x\) by
\[
\mathcal{P}_{x}
:=
\left\{
\Pi\in\operatorname{Gr}_{2}(\mathbb{R}^{n+1})
\;\middle|\;
x\in\Pi
\text{ and }
(X\cap\Pi)\setminus\{x\}\neq\emptyset
\right\}.
\]
We claim that \(\#\mathcal{P}_x=1\). Since \(x\) and any \(y\in X\setminus\{x\}\) are linearly independent, the \(2\)-dimensional subspace of \(\mathbb{R}^{n+1}\) spanned by \(x\) and \(y\) belongs to \(\mathcal{P}_x\). Thus \(\mathcal{P}_x\neq\emptyset\). Now let \(\Pi_1,\Pi_2\in\mathcal{P}_x\). Then
\[
(X\cap\Pi_1\cap\Pi_2)\setminus\{x\}\neq\emptyset
\Longrightarrow
\Pi_1=\Pi_2.
\]
Indeed, if \(X\cap\Pi_1\cap\Pi_2\) contains some \(y\neq x\), then \(x\) and \(y\) are linearly independent. Hence both \(2\)-dimensional subspaces \(\Pi_1\) and \(\Pi_2\) have the common basis \(x,y\), and therefore coincide. It follows that
\[
X=\{x\}\sqcup\bigsqcup_{\Pi\in\mathcal{P}_x}\bigl((X\cap\Pi)\setminus\{x\}\bigr).
\]
For each \(\Pi\in\mathcal{P}_x\), Proposition~\ref{prop_subspaceV_cap_subquandle} shows that \(X\cap\Pi\) is a finite subquandle of \(S^n_{\mathbb{R}}\). By Proposition~\ref{prop_sum_subquandle_elements=0}, we have
\[
\sum_{y\in (X\cap\Pi)\setminus\{x\}} y=-x.
\]
Therefore,
\[
\sum_{y\in X}y
=
x+\sum_{\Pi\in\mathcal{P}_x}\sum_{y\in (X\cap\Pi)\setminus\{x\}} y
=
x+\sum_{\Pi\in\mathcal{P}_x}(-x)
=
(1-\#\mathcal{P}_x)x.
\]
On the other hand, Proposition~\ref{prop_sum_subquandle_elements=0} again implies that \(\Sigma_X=\sum_{y\in X}y=0\). Hence \(\#\mathcal{P}_x=1\).

If \(\Pi\in\mathcal{P}_x\), then \(X=X\cap\Pi\). By Remark~\ref{rem_antipodalpoint_subquandle_S^n}, it follows that \(X\) is isomorphic to \(\mathcal{R}_{\#X}\), and in particular \(\#X=\#(X\cap\Pi)\) is odd.
\end{proof}
Theorem \ref{main_theorem} follows from Proposition~\ref{prop_main1_root_system} and Proposition~\ref{prop_main2_dihedral_quandle}.

\section*{Acknowledgements}
The first author would like to express gratitude to Professor Hideyuki Ishi. Theorem~\ref{main_theorem} originated from his participation in the ``8th Tunisian-Japanese Conference Geometric and Harmonic Analysis on Homogeneous Spaces and Applications," to which Professor Ishi kindly invited him.

The authors would like to express their sincere congratulations to Professor Hiroyuki Ochiai on the occasion of his 60th birthday. Part of this work was presented at the conference held in his honor.

The authors thank Michiko Yonemura for drawing Fig. \ref{pic_def_quandle_operation}.

This work was partly supported by MEXT Promotion of Distinctive Joint Research Center Program JPMXP0723833165 and Osaka Metropolitan University Strategic Research Promotion Project (Development of International Research Hubs).
\bibliography{bibliography}

@article{Azcan2010finite,
  title={Finite subquandles of sphere},
  author={{\"O}zdemir, N{\"u}lifer and Azcan, H{\"u}seyin},
  journal={Turkish Journal of Mathematics},
  volume={34},
  number={2},
  pages={293--304},
  year={2010}
}

@article{AzcanFenn1994,
  title={Spherical representations of the link quandles},
  author={Azcan, Hiiseyin and Fenn, Roger},
  journal={Turkish J. of Mathematics},
  volume={18},
  pages={102--110},
  year={1994}
}

@article{ChenNagano1988riemannian,
  title={A Riemannian geometric invariant and its applications to a problem of Borel and Serre},
  author={Chen, Bang-Yen and Nagano, Tadashi},
  journal={Transactions of the American Mathematical Society},
  pages={273--297},
  year={1988},
  publisher={JSTOR}
}

@book{humphreys1992reflection,
  title={Reflection groups and Coxeter groups},
  author={Humphreys, James E},
  number={29},
  year={1992},
  publisher={Cambridge university press}
}

@article{IshiharaTamaru2016flat,
  title={Flat connected finite quandles},
  author={Ishihara, Yoshitaka and Tamaru, Hiroshi},
  journal={Proceedings of the American Mathematical Society},
  volume={144},
  number={11},
  pages={4959--4971},
  year={2016}
}

@article{joyce1982classifying,
  title={A classifying invariant of knots, the knot quandle},
  author={Joyce, David},
  journal={Journal of Pure and Applied Algebra},
  volume={23},
  number={1},
  pages={37--65},
  year={1982},
  publisher={Elsevier}
}

@book{kamada2017surface,
  title={Surface-knots in 4-space},
  author={Kamada, Seiichi},
  volume={22},
  year={2017},
  publisher={Springer}
}

@article{takasaki1943abstraction,
  title={Abstraction of symmetric transformations},
  author={Takasaki, Mituhisa},
  journal={Tohoku Mathematical Journal, First Series},
  volume={49},
  pages={145--207},
  year={1943},
  publisher={Mathematical Institute, Tohoku University}
}

@article{matveev1982distributive,
  title={Distributive groupoids in knot theory},
  author={Matveev, Sergei Vladimirovich},
  journal={Mathematics of the USSR-Sbornik},
  volume={47},
  number={1},
  pages={73--83},
  year={1982}
}

@book{nosaka2017books,
  title={Quandles and topological pairs: Symmetry, knots, and cohomology},
  author={Nosaka, Takefumi},
  year={2017},
  publisher={Springer}
}
\bibliographystyle{plain}

\end{document}